\documentclass[12pt]{article}
\usepackage{amsmath}
\usepackage{amssymb}
\usepackage{amsthm}

\newtheorem{definition}{Definition}
\newtheorem{theorem}{Theorem}

\newtheorem{remark}{Remark}

\newtheorem{corollary}{Corollary}
\newtheorem{example}{Example}

\newcommand{\bC}{{\mathbb C}}

\begin{document}

\title{Stratification and coordinate systems  
 for the moduli space of rational functions}
\author{Masayo Fujimura and Masahiko Taniguchi}
\date{preprint}

\maketitle

\abstract{
In this note, we give a new simple system of 
global parameters on the 
moduli space of rational functions, and clarify 
the relation to the parameters indicating 
location of fixed points and the indices at them. 
As a byproduct, we solve a conjecture of Milnor affirmatively.
}

\section{Introduction}

Let ${\rm Rat}_d$ be the set of all rational 
functions of degree $d>1$, 
and ${\rm M}_d$ the set of all 
M\"obius conjugacy classes 
of elements in $ {\rm Rat}_d $, 
which is called the {\em moduli space  of rational functions of 
degree} $d$.

Here it is a fundamental problem 
to give a good system of parameters on ${\rm M}_d$. 
And McMullen showed in \cite{Mc} that, outside the Latt\'e loci, 
every multiplier spectrum at periodic points corresponds 
to a finite number of points in ${\rm M}_d$. 
This result is epoch-making, and many researches have been done 
on the system of multipliers, or indices, at periodic points. 
Among other things, the following example is well-known.

\begin{example}[\cite{M2}]\label{ex:milnor}
{\rm 
When $d=2$, there are $3$  fixed points counted 
including multiplicities, the indices of
which satisfy a single simple relation, called 
Fatou's index formula. Hence, we can consider a
map $\Phi_2:{\rm M}_2\to \bC^2$ induced by two of three
fundamental symmetric functions of 
multipliers at fixed points.
This map 
$\Phi_2$ is bijective, and hence gives a 
coordinate system for ${\rm M}_2$.}
\end{example}

\begin{remark}
In the case of polynomials, 
the set of multipliers, or indices, at
fixed points gives an interesting system 
of parameters on the moduli space of 
polynomials. For the details, see \cite{F} and \cite{FT}.
\end{remark}

Clearly, the multipliers at fixed points only are not enough 
to parametrize the moduli space ${\rm M}_d$ when $d>2$. 
But, it seems difficult to find a suitable set of 
multipliers at periodic points for obtaining a good 
system of global parameters. 
On the other hand, in the case of polynomials, 
the set of monic centered ones is 
often used as a virtual set of representatives of points in
the moduli space ${\rm MPoly}_d$ of polynomials of degree $d$, 
and it is well-known that the coefficients of them give 
a useful set of
 parameters on ${\rm MPoly}_d$, which in particular induces the 
complex orbifold structure of ${\rm MPoly}_d$.
We give, in \S 2, a family of rational functions 
whose coefficients give a good system of 
parameters on ${\rm M}_d$  
in a similar sense as in 
the case of the family of monic centered polynomials.

In \S 3, we investigate the correspondence between these 
coefficient parameters 
and the union of the set of the indices and 
location of fixed points, which gives a candidate of 
an important subsystem of parameters on ${\rm M}_d$. 
Here, the overlap type of 
fixed points naturally gives a stratification of ${\rm M}_d$. 
We introduce a natural system of coordinates on each stratum. 
As a byproduct, we give 
an affirmative answer to a conjecture of Milnor proposed 
in the book \cite{M}.

\section{A normalized family of rational functions}

A general form of a rational function of degree $d$ is 
\[
\frac{P(z)}{Q(z)}
\]
with polynomials $P(z)$ and $Q(z)$ of degree at most $d$, where 
$P(z)$ and $Q(z)$ have no common non-constant factors and 
one of them has $d$ as the degree. 
To consider the moduli space ${\rm M}_d$, we may 
assume without loss of generality that $Q(z)$ is 
of degree $d$, and that the resultant ${\rm Resul}(P,Q)$  
of $P(z)$ and $Q(z)$ does not vanish. 
Also it imposes no restriction to assume 
that $Q(z)$ is monic. We call such 
a rational function satisfying 
the above conditions a {\em canonical function}.

\begin{definition}
{\rm 
The {\em canonical family} $C_d$ of rational functions of degree $d$ 
is defined
as the totality of canonical functions of degree $d$ as above:
\[
\left\{R(z) = \frac{P(z)}{Q(z)} \in {\rm Rat}_d \Biggm| {\rm deg}\, Q = d, \
{\rm Resul}(P,Q)\not = 0,
\mbox{ $Q$ is monic}\right\}.
\]
Moreover, writing 
\[
P(z) = a_dz^d + \cdots + a_0, \qquad 
Q(z) = z^d + b_{d-1}z^{d-1} +\cdots + b_0,
\]
we call the vector 
$(a_d, \cdots, a_0, b_{d-1}, \cdots, b_0)$ the system of 
{\em coefficient parameters} for $C_d$.
}
\end{definition}

Every point in ${\rm M}_d$ contains an element in $C_d$ as a 
representative. 
On the other hand, since ${\rm M}_d$ is $(2d-2)$-dimensional, while 
the dimension of $C_d$ is $2d+1$, we can 
consider to impose three normalization 
conditions on elements in $C_d$. Here we impose 
\[
a_0 = 0, \quad b_1 = -1, \mbox{ and } \ b_0 = 1.
\]
We call a rational function in $C_d$ satisfying 
 these conditions a {\em normalized function}.

\begin{definition}
{\rm 
We call the family consisting of all normalized functions in
 $C_d$ the {\em normalized family} of degree $d$, and denoted by 
 $N_d$.
}
\end{definition}

More explicitly,
\[
  N_d= \left\{ \frac{a_dz^d+ \cdots + a_1z}
{z^d+b_{d-1}z^{d-1}+ \cdots + b_2z^2-z+1} \in C_d \right\},
\]
and we call the vector $(a_d,\cdots, a_1,b_{d-1},\cdots,b_2)$ 
the system of {\em coefficient parameters} for $N_d$.
Here, we can show that $N_d$ 
is an ample family of rational functions for every $d$.

\begin{example}
{\rm 
When $d=2$, the natural projection of $N_2$ to ${\rm M}_2$ is 
surjective.
To see this, it suffices to show that every possible
set of multipliers $\{m_1,m_2,m_3\}$ at fixed points corresponds
to a rational function in $N_2$ 
(cf. Example \ref{ex:milnor}).

First, if the set is  $\{1,1,1\}$, then a corresponding rational 
function in $N_2$ is uniquely determined 
(cf. Example \ref{ex:overlap}) and is  
\[
  R(z)=\frac{-z^2+z}{z^2-z+1}. 
\]
If the set is  $\{1,1,m\}$ with $m\not=1$,
 then a corresponding rational function is
\[
  R(z)=\frac{z(mz+p)}{p(z^2-z+1)} 
\]
with a solution $p$ of $ p^2+(m+1)p+m^2=0. $

Next, in the remaining cases, the set $ \{m_1,m_2,m_3\} $ of
multipliers satisfies that
$m_j\not=1$ $(j=1,2,3)$ and Fatou's index formula
\[
    \frac{1}{1-m_1}+\frac{1}{1-m_2}+\frac{1}{1-m_3}=1.
\]
Here if the set is $\{0,0,2\}$, we can see that a corresponding rational
function is
\[
  R(z)=\frac{(3/2)z^2}{z^2-z+1}. 
\]
And otherwise, we can choose $m$ and $m'$ among $\{m_1,m_2,m_3\}$
so that
\[
    m'\not\in \{0, \pm i/\sqrt{3}\}, \
    mm'-1\neq0 ,\ \mbox{ and  } \  m+m'-2 \neq 0,
\]
which are assumed to be  $m_1$ and $m_2$, respectively.
Then the equation
\[
   (-m_1^2+3m_1-3)p^2+(2m_2m_1-3m_2-1)p-m_2^2=0
\]
has a non-zero solution $p$.
With this $p$, we see that a corresponding rational function is
\[
    R(z)=\frac{-\bigl((m_1-2)p-m_2\bigr)
          \bigl((m_1p+1)z+(m_1^2-2m_1)p-m_2m_1\bigr)z} 
          {p\bigl((m_1-1)p-m_2+1\bigr)(z^2-z+1)}.
\]
Here, $ (m_1-1)p-m_2+1\not=0 $ from the assumption, 
and we conclude the assertion when $d=2$.

Note that, in terms of the
fundamental symmetric functions 
\[
   \sigma_1=m_1+m_2+m_3,\ \ \sigma_2=m_1m_2+m_1m_3+m_2m_3, \quad 
   \mbox{and}\quad 
   \sigma_3=m_1 m_2 m_3,
\]
the natural projection of 
$ N_2\cong \{(a_2,a_1)\mid a_2^2+a_1a_2+a_1^2\not=0\}$ to ${\rm M}_2$ 
is given by
\begin{align*}
  \sigma_1 &=\frac{2a_2^2+a_1^2a_2+a_1^3-2a_1^2+3a_1}
                  {a_2^2+a_1a_2+a_1^2},\\
  \sigma_2 &= \frac{-(a_1^2-2a_1)a_2^2+(a_1-2)a_2-2a_1^3+4a_1^2-4a_1+3}
               {a_2^2+a_1a_2+a_1^2},\\
  \sigma_3 & =\sigma_1-2.
\end{align*}
}
\end{example}

In general, we obtain the following.

\begin{theorem}
For every $d\geq 2$, 
the natural projection of $N_d$ to ${\rm M}_d$ is surjective.
\end{theorem}

\begin{proof}
The assertion for the case that $d=2$ is shown in the above example.
When $d= 3$, we can show the assertion by direct calculations 
using a symbolic and algebraic computation system,
the detail of which is contained in \S 4 for the sake of 
readers' convenience. 
So, we assume that $d\geq 4$ in the sequel of the proof.

Let $x$ be a point of ${\rm M}_d$ and 
$R(z)$ a rational function of degree $d$ contained in the 
M\"obius conjugacy class $x$. Then we may assume that 
$R(z)$ is canonical and $R(0)=0$, 
by taking a M\"obius conjugate 
of $R(z)$ if necessary, which implies 
in particular that 
\[
  a_0 =0 \quad \mbox{and}\quad  b_0\not= 0. 
\]

Next, if we take conjugate of $R(z)$ by a translation 
$L(z) = z+ \alpha$. 
Then we have 
\[
  L^{-1}\circ R\circ L(z) = 
    \frac{\left(a_d(z+\alpha)^d+ \cdots +a_0 \right) 
    - \alpha \left( (z+\alpha)^d+ \cdots +b_0 \right) }
    {(z+\alpha)^d+ \cdots +b_0}, 
\]
which we write as
\[
  \frac{\tilde{a}_dz^d + \cdots + \tilde{a}_0}
  {z^d + \tilde{b}_{d-1}z^{d-1} +\cdots + \tilde{b}_0}.
\]
Here, if $\alpha$ is a fixed point of $R(z)$, then 
\[
  \tilde{a}_0 =0, \quad \mbox{and}\quad   \tilde{b}_0\not= 0. 
\]
Also, taking as $\alpha$ one, say $\zeta_R$, of 
fixed points of $R(z)$ with the largest 
multiplicities, we may assume that 
$R(z)$ has no non-zero fixed points 
with multiplicity $d$. 
Moreover, if $0$ is a 
non-simple fixed point of $L^{-1}\circ R\circ L(z)$, then 
\[
  \tilde{a}_1= \tilde{b}_0.
\]
And hence if $\tilde{a}_1\not= \tilde{b}_0$, 
then every fixed points of $R(z)$ is simple,
and there is a non-zero fixed point $\zeta_R$ of $R(z)$ such that 
$\tilde{b}_1\not=0$. 
Indeed, letting $\{\zeta_1,\cdots, \zeta_d\}$ be the set of non-zero 
fixed points of $R(z)$, we consider conjugates of $R(z)$ by 
$L_k(z) = z+\zeta_k$. Then 
\[
  \tilde{b}_1 = d \zeta_k^{d-1}+ 
   (d-1)b_{d-1} \zeta_k^{d-2}+ \cdots + b_1
\] 
can not be $0$ for all $k$. Also repeating such change of 
fixed points again if necessary, 
we can further assume that there are neither circles nor lines 
in $\bC-\{0\}$ 
which contain all non-zero fixed points, since we have assumed that 
 $d\geq 4$.

Thus we may assume from the beginning that 
{\em 
$a_0 =0$, $b_0\not= 0$, and 
$(db_0-a_1)z+ b_1$ is not constantly $0$, 
$R(z)$ has no non-zero fixed points 
with the multiplicities $d$, and if $R(z)$ has simple 
fixed points only, 
then there are neither circles nor lines in $\bC-\{0\}$ 
which contain all non-zero fixed points. }

Now, set 
\[
  T(z) = \frac{z}{pz+q} \qquad (q\not=0).
\]
Then we have 
\[
  T^{-1}\circ R\circ T(z) = 
  \frac{q(a_dz^d+ \cdots +a_1z(pz+q)^{d-1})}
  {-p(a_dz^d+ \cdots +a_1z(pz+q)^{d-1}) + (z^d+ \cdots +b_0(pz+q)^d)}.
\]
The constant term of the numerator remains to be $0$, 
and the coefficients of $z^d$ 
in the numerator and the denominator change to 
\begin{align*}
  a_d^*(p,q) &= q(a_d+ \cdots +a_1p^{d-1}) \quad \mbox{and} \\
  b_d^*(p) &= -p(a_d+ \cdots +a_1p^{d-1}) 
                         + (1+ \cdots +b_0p^d),
\end{align*}
respectively.
If $b_d^*(p)\not=0$, divide both of the numerator and the denominator 
of the conjugate $T^{-1}\circ R\circ T(z)$ by $b_d^*(p)$. Then 
the coefficients $a_1$, $b_0$ and $b_1$, for instance, 
change to 
\begin{align*}
  a_1(p,q) &= \displaystyle{\frac{a_1q^d}{b_d^*(p)}}, \quad 
  b_0(p,q) = \displaystyle{\frac{b_0q^d}{b_d^*(p)}}, \\
  b_1(p,q) &= \displaystyle{\frac{-a_1pq^{d-1}+ b_1q^{d-1}
                             + db_0pq^{d-1}}{b_d^*(p)}}.
\end{align*} 
Also, the condition $b_1(p,q)/b_0(p,q) = -1$ implies that
\[
  q=q(p) =-\frac{(db_0-a_1)p+b_1}{b_0}.
\]

First, if $b_0=a_1$, 
then $b_0(p,q(p))$ is a rational function of $p$ such that 
the degrees of the numerator and the denominator are 
exactly $d$ and not greater than $d-1$, 
respectively. Hence there is a finite $p$ with $b_0(p,q(p))=1$. 
Next, if $db_0=a_1$, 
then $b_0(p,q(p))$ is a rational function of $p$ such that 
the degree of the denominator is exactly $d$ and 
the numerator is a non-zero constant.
Hence there is a finite $p$ with $b_0(p,q(p))=1$. 
Finally, if otherwise, namely, if $b_0\not=a_1$ and $db_0\not=a_1$, 
then the degrees of 
the numerator and the denominator are exactly $d$ and 
$R(z)$ has simple fixed points only. We write non-zero fixed points of 
$R(z)$ as
$\{\zeta_k\}_{k=1}^d$. Suppose that $b_0(p,q(p))$ 
can take the value $1$ at $\infty$ only.
Then with some non-zero constant $C$, 
\[
  b_0(p,q(p))= 1 + \frac{C}{b_d^*(p)},
\]
which implies that $\{1/\zeta_k\}_{k=1}^d$ lie on the same circle, 
for $b_d^*(p) = \prod_{k=1}^d\, (1-\zeta_kp)$.
But then, $\{\zeta_k\}_{k=1}^d$ should be 
on a same circle or a line not containing 
$0$, which contradicts to one of the assumptions from the beginning.
Hence we conclude also in this case that 
there is a finite $p$ such that $b_0(p,q(p))=1$.

Thus we obtain a $T(z)$ such that
$T^{-1}\circ R\circ T(z)$  belongs to $N_d$ if $d\geq 4$, 
and the proof is now complete.
\end{proof}

For a generic point of ${\rm M}_d$, 
there are only a finite number of 
rational functions in $N_d$ belonging to 
the point, as is seen from the proof of Theorem 1. 
On the other hand, some points of ${\rm M}_d$ can blow up in 
$N_d$ as in the following example.

\begin{example}
{\rm 
Set 
\[
  R(z)=\frac{-3z^3-4z^2-2z}{z^3-z-1}.
\]
Then $R(z)$ has a simple fixed point at $0$, and 
one with multiplicity $3$ at $-1$.

As in the proof of Theorem 1, letting 
\[
  T(z)=\frac{z}{pz-1-p} \quad(p\neq -1),
\]
set $ R_p(z)=T^{-1}\circ R\circ T(z)$.
Then we have
\[
  R_p(z)=
  \frac{(2p^2+4p+3)z^3+(-4p^2-8p-4)z^2+(2p^2+4p+2)z}
  {(p^2+2p+1)z^3+(-p^2-2p)z^2+(-p^2-2p-1)z+p^2+2p+1}.
\]
Hence if we set $ \tilde{p}=1/(p^2+2p+1) $, 
\[
  R_p(z) = \tilde{R}_{\tilde{p}}(z)=
    \frac{(\tilde{p}+2)z^3-4z^2+2z}{z^3+(\tilde{p}-1)z^2-z+1}.
\]
Thus $ \tilde{R}_{\tilde{p}}(z) $ belongs to $N_3$, and 
represents the same point of ${\rm M}_3$ for every non-zero $\tilde{p}$.
Indeed, every $ \tilde{R}_{\tilde{p}}(z) $ is conjugate to 
$ \tilde{R}_{1}(z) $ by 
\[
S(z) = \frac{z}{(1-\tilde{p}^{1/2})z + \tilde{p}^{1/2}}.
\]
}
\end{example}

\section{A stratification of the moduli space}

Every rational function $R(z)=P(z)/Q(z)$ 
of degree $d$ not fixing $\infty$ 
can be written also as 
\[
  R(z) = z - \frac{\hat{P}(z)}{Q(z)}
\]
with monic polynomials $\hat{P}(z)$ and $Q(z)$ of degree $d+1$ 
and $d$, respectively. Using this representation, we have another 
system of parameters, some of which are fixed points of $R(z)$.

\begin{definition}
{\rm
Let
\[
  R(z) = z - \frac{\hat{P}(z)}{Q(z)},
\]
with
\[
  \hat{P}(z) = zQ(z)- P(z)=\prod_{j=1}^p\, (z-\zeta_j)^{n_j} 
   \qquad (\zeta_j\in \bC),
\]
where $\zeta_j$ are mutually distinct and
 $n_j$ are positive integers which satisfy 
\[
\sum_{k=1}^p n_k=d+1.
\] 
Then we call the set 
$\{n_1, \cdots, n_p\}$ the {\em overlap type}
of fixed points of $R(z)$. 

We set 
\[
  C\{n_1, \cdots, n_p\} =
     \bigl\{R(z)\in C_d  \bigm| 
        \mbox{ the overlap type is }\{n_1, \cdots, n_p\}\bigr\}
\]
and call it the {\em $\{n_1, \cdots, n_p\}$-locus of $C_d$}.
The subset
\[
  C_d' =
    \bigl\{R(z)\in C_d 
       \bigm| \mbox{ the overlap type is not }
      \{1, \cdots, 1\}\bigr\}
\]
of $C_d$ is called the {\em overlap locus} of $C_d$. 

Similarly, we can define the 
{\em $\{n_1, \cdots, n_p\}$-locus of $N_d$}
by setting 
\[
    N\{n_1, \cdots, n_p\} =
     \bigl\{R(z)\in N_d  \bigm| \mbox{ the overlap type is }
      \{n_1, \cdots, n_p\}\bigr\}.
\]
Also the subset
\[
    N_d' =
      \bigl\{R(z)\in N_d
        \bigm| \mbox{ the overlap type is not }
        \{1, \cdots, 1\}\bigr\}
\]
of $N_d$ is called the {\em overlap locus} of $N_d$. 
}
\end{definition}

Since the overlap type of fixed points is 
invariant under M\"obius 
conjugation, Theorem 1 implies the following result.

\begin{corollary}
Let ${\rm M}_d'$ be the subset of all points of ${\rm M}_d$ 
represented by 
rational functions having non-simple fixed points. Then 
the natural projection $\pi$ of $N_d'$ to ${\rm M}_d'$ 
is surjective for every $d\geq 2$.
\end{corollary}

\begin{definition}
{\rm 
The image of every 
$\{n_1, \cdots, n_p\}$-locus of $N_d$ 
by $\pi$ is called 
 the {\em $\{n_1, \cdots, n_p\}$-stratum of ${\rm M}_d$}, and 
 denoted by $M\{n_1, \cdots, n_p\}$.
The resulting stratification 
of  ${\rm M}_d$ is called the 
{\em overlap type stratification}. 
}
\end{definition}

\begin{remark}
The above loci are defined by algebraic equations 
(cf. Example \ref{ex:overlap} and \ref{ex:overlap2}), and 
hence a Zariski open subset of complex algebraic sets in $C_d$
 and in $N_d$ (with respect to the system of coefficient parameters). 
For instance, 
\[
   C_d' = \left\{ \frac{\hat{P}(z)}{Q(z)}
      \in C_d \bigm|  {\rm Discr}(\hat{P}) = 0\right\}.
\]
\end{remark}

\begin{example}\label{ex:overlap}
{\rm
In the case of $d=2$,
\begin{align*}
   C\{3\} &\cong \biggl\{(a_2,a_1,a_0,b_1,b_0)\biggm|
             \begin{array}{l}
                a_1 = b_0-(b_1-a_2)^2/3, \\
                a_0 = -(b_1-a_2)^3/27
             \end{array} \biggr\}, \\
   C_2' &\cong \biggl\{(a_2,a_1,a_0,b_1,b_0)\biggm|
            \begin{array}{l}
             -27a_0^2 + a_0\left\{ 4(b_1-a_2)^3-18(b_0-a_1)(b_1-a_2)\right\} \\
             \ \ +(a_1-b_0)^2(b_1-a_2)^2+4(a_1-b_0)^3=0
            \end{array}\biggr\}, \\
   N\{3\} & \cong \bigl\{(-1,1,0,-1,1) \bigr\}, \\
   N_2' & \cong \bigl\{ (a_2,a_1,0,-1,1)\bigm| \ 
             a_1-1=-(a_2+1)^2/4 \quad \mbox{or} \quad a_1=1 \bigr\}. 
\end{align*}
}
\end{example}
 
\begin{example}\label{ex:overlap2}
{\rm
In the case of $d=3$,
\begin{multline*}
   C\{4\}\cong \biggl\{(a_3,a_2,a_1,a_0, b_2,b_1,b_0)\\
           \quad \biggm|
             \begin{array}{l}
               a_2 =b_1 -3(b_2-a_3)^2/8,\quad
               a_1 = b_0 -(b_2-a_3)^3/16, \\
               a_0 = -(b_2-a_3)^4/256
             \end{array} \biggr\},
\end{multline*}
\begin{multline*}
  C_3'\cong \biggl\{(a_3,a_2,a_1,a_0, b_2,b_1,b_0)\biggm| \ 
                \mbox{D}=0\biggr\}, \\
\end{multline*}
where
\begin{align*}
  \mbox{D}
   & =256a_0^3
       +a_0^2\bigl\{128(b_1-a_2)^2-144(b_2-a_3)^2(b_1-a_2)+27(b_2-a_3)^4\\
   & +192(b_2-a_3)(b_0-a_1)\bigr\}
      +a_0\bigl\{16(b_1-a_2)^4-4(b_2-a_3)^2(b_1-a_2)^3\\
   &   -80(b_0-a_1)(b_2-a_3)(b_1-a_2)^2
       +18(b_0-a_1)((b_2-a_3)^3+8(b_0-a_1))(b_1-a_2)\\
   &     -6(b_0-a_1)^2(b_2-a_3)^2\bigr\}
      +(b_0-a_1)^2(4(b_1-a_2)^3-(b_2-a_3)^2(b_1-a_2)^2\\
   &              -18(b_0-a_1)(b_2-a_3)(b_1-a_2)
                 +(b_0-a_1)(4(b_2-a_3)^3+27(b_0-a_1)))
\end{align*}
\begin{align*}
  N\{4\} \cong \bigl\{(c,-1,1,0, c,-1,1)\bigm| c\in \bC \bigr\},
\end{align*}
and
\begin{multline*}
  N_3' \cong \biggl\{(a_3,a_2,a_1,0, b_2,-1,1)\\
          \quad\biggm| 
           \begin{array}{l}
               -27(a_1-1)^2+ (a_1-1)\bigl(4(b_2-a_3)^3
               +18(a_2+1)(b_2-a_3)\bigr) \\ \ \ 
               +(a_2+1)^2(b_2-a_3)^2+4(a_2+1)^3=0  
               \quad \mbox{or} \quad   a_1=1 
           \end{array}\biggr\}.
\end{multline*}
}
\end{example}
 
 \medskip
On the other hand, it is well-known that the denominator $Q(z)$ 
of $R(z)$ in $C\{n_1, \cdots, n_p\}$ 
can be represented uniquely as
\[
  Q(z) = \sum_{k=1}^p\, 
    \biggl\{
    \Bigl(\sum_{n=0}^{n_k-1}\, \alpha_{k,n_k-n}{(z-\zeta_k)^{n}}\Bigr)
  \prod_{j\not=k} \, (z-\zeta_j)^{n_j}\biggr\}.
\]  
In other words, $Q(z)/\hat{P}(z)$ has a unique partial fractions decomposition
\[
 \frac{\alpha_{1,n_1}}{(z-\zeta_1)^{n_1}}+\cdots
              +\frac{\alpha_{1,1}}{z-\zeta_1} 
              +\frac{\alpha_{2,n_2}}{(z-\zeta_2)^{n_2}}+\cdots
              +\frac{\alpha_{p,1}}{z-\zeta_p}.
\]
Here, the assumptions imply that $\alpha_{k,n_k}\not=0$ 
for every $k$ and 
\[
  \sum_{k=1}^p\, \alpha_{k,1} = 1.
\]

\begin{definition}
{\rm 
The set  $\{\zeta_k\}$ of fixed points and the set 
$\{\alpha_{k,\ell}\}$ of coefficients 
 give a system of parameters for $C\{n_1, \cdots, n_p\}$, 
 and is called
the system of {\em decomposition parameters} 
for $C\{n_1, \cdots, n_p\}$.
}
\end{definition}

\begin{theorem}
Set
\begin{align*}
  \tilde{E}\{n_1, \cdots, n_p\} =
    \biggl\{ 
  & (\zeta_1, \cdots, \zeta_p,
        \alpha_{1,1}, \cdots, \alpha_{1,n_1},
        \alpha_{2,1}, \cdots, \alpha_{p,n_p}) \in \bC^{d+p+1}\\
  & \biggm|  \sum_{k=1}^p\, \alpha_{k,1}= 1, \quad 
       \alpha_{k,n_k}\not=0 \quad (k=1,\cdots, p) \biggl\}.
\end{align*}
Then the natural projection $\Pi$ of 
$\tilde{E}\{n_1, \cdots, n_p\}$ to 
$C\{n_1, \cdots, n_p\}$ (with respect to the system of 
coefficient parameters) is a holomorphic surjection.

Moreover, $C\{n_1, \cdots, n_p\}$ has a 
complex manifold structure such that $\Pi$ is
a finite-sheeted holomorphic covering projection.
\end{theorem}

We call $\tilde{E}\{n_1, \cdots, n_p\}$ the 
{\em marked $\{n_1, \cdots, n_p\}$-parameter domain}.

\proof
Since $\Pi$ is a polynomial map, it is holomorphic. 
To show other assertions, note that 
the defining domains of the system of decomposition 
parameters is the product space 
\[
  \prod_{n=1}^{d+1}\, C_{N_n}(\bC^{n+1}),
\]
where $C_m(\bC^n)$ is the configuration space of $m$ distinct vectors in $\bC^n$ and $N_n$ is the number of $\ell$ 
with $n_\ell= n$. In particular, $N_n>0$ only if 
\[
  \min \{n_1,\cdots, n_p\} \leq n\leq \max \{n_1,\cdots, n_p\},
\]
the set $ \{(n,1),\cdots, (n,N_n)\}$ 
is empty if there are no $\ell$ with $n_\ell=n$, 
and 
\[
\sum_{n=1}^{d+1}\, nN_n = d+1.
\]
The coordinates of the product space 
can be written explicitly as follows;
\begin{align*}
  & E\{n_1, \cdots, n_p\}\\
  & \begin{aligned}
      = \Biggl\{\Bigl(
     &
      \bigl\{ \left(\zeta_{1,1}, \alpha_{(1,1),1}\right), \cdots, 
      \left(\zeta_{1,N_1}, \alpha_{(1,N_1),1}\right) \bigr\}, 
       \  \cdots\cdots   \\
     & 
       \bigl\{ 
       \left(\zeta_{d+1,1}, \alpha_{(d+1,1),1},\cdots, 
       \alpha_{(d+1,1),d+1}\right), \cdots, \\
     &  \quad
       \left(\zeta_{d+1,N_{d+1}}, \alpha_{(d+1,N_{d+1}),1},\cdots, 
       \alpha_{(d+1,N_{d+1}),d+1} \right)
       \bigr\} \Bigr)\in \prod_{n=1}^{d+1} C_{N_n}(\mathbb{C}^{n+1}) \\
     &  
       \Biggm| \ 
       \sum_{k=1}^{d+1}\biggl(\sum_{j=1}^{N_k}\, \alpha_{(k,j),1}\biggr)= 1,
       \quad 
       \alpha_{(k,\ast),k}\not=0 \quad (k=1, \cdots, p)\Biggr\},
   \end{aligned}
\end{align*}
where all $\zeta$s are mutually disjoint as before.

Now the map $\Pi$ is factored through by the canonical  
finite-sheeted holomorphic covering projection $\sigma$ 
of $\tilde{E}\{n_1, \cdots, n_p\}$ to 
$E\{n_1, \cdots, n_p\}$ and the natural holomorphic bijetion 
$\iota$ of  
$E\{n_1, \cdots, n_p\}$ to $C\{n_1, \cdots, n_p\}$:
\[
\Pi = \iota \circ \sigma.
\]
In particular, $\iota$ induces the desired 
complex manifold structure on 
$C\{n_1, \cdots, n_p\}$.
\qed

\begin{remark}
On the non-overlap locus $C\{1, \cdots, 1\}= C_d-C_d'$, 
$\{\alpha_{k,1}\}_{k=1}^{d+1}$ in the system of decomposition parameters
are nothing but the indices at the fixed points 
$\{\zeta_k\}_{k=1}^{d+1}$, which 
implies the assertion of Problem 12-d in \cite{M}. 
\end{remark}

\begin{corollary}
If the location and the overlap type of fixed points and the 
indices at them are fixed, then
the resulting subset of 
$ C\{n_1, \cdots, n_p\}$ has a natural complex manifold structure 
of dimension $d+1-p$.
\end{corollary}

\proof
By Theorem 3, we need only to note that 
\[
  \dim_{\bC}\, C\{n_1, \cdots, n_p\} = d+p.
\]
\qed

\medskip
This corollary gives the affirmative answer to a 
conjecture of Milnor stated in Remark below 
 Problem 12-d \cite[p.152]{M}.
 
\section{The proof of Theorem 1 for the case that $d=3$}

Even in the case that $d=3$, the arguments of the proof of 
Theorem 1 can be applied,
but we can not exclude the case that $R(z)$ has $4$ simple fixed 
points $0, w_1,w_2,w_3$ such that $1/w_1,1/w_2,1/w_3$ 
lie on the same circle. So, we will treat this case 
by direct calculation using a symbolic and algebraic computation 
system (cf.  \cite{cox-ideal}, \cite{cox-using}).

For this purpose, let
$0, w_1,w_2,w_3 $ be the set of simple fixed points of
a given $R(z)$ of degree $3$ (having simple fixed points only).
We may assume that 
the denominator of which has the form
$z^3+b_2z^2+b_1z+b_0$ with $ b_0\neq 0 $ as before.
Let 
\[
  T(z)=\frac{z}{pz+q}\quad (q\neq 0),
\] 
and take the conjugate of $ R(z) $ by $ T(z) $. Then 
the coefficients  $1, b_1$, and $b_0$ in the denominator 
$z^3+b_2z^2+b_1z+b_0$ change to
\[
\begin{array}{l}      
  b_3^*(p)=w_3w_2w_1p^3-((w_2+w_3)w_1+w_3w_2)p^2+(w_1+w_2+w_3)p-1,\\
      b_1^*(p,q)=(w_3w_2w_1-2b_0)q^2p-b_1q^2, \mbox{ and }\\
      b_0^*(q)=-b_0q^3.
       \end{array} 
\]
So the condition $ b_1^*(p,q)/b_0^*(q)=-1 $ implies that 
\[
   q=\frac{(w_3w_2w_1-2b_0)p-b_1}{b_0} 
\]
and the condition $ b_0^*(p,q)/b_3^*(p)=1 $ is the equation
\begin{multline}\label{eq:A}
  (-w_1^3w_2^3w_3^3+6b_0w_1^2w_2^2w_3^2-13b_0^2w_1w_2w_3+8b_0^3)p^3\\ 
   +((3w_1^2w_2^2w_3^2-12b_0w_1w_2w_3+12b_0^2)b_1 
    +(b_0^2w_2+b_0^2w_1)w_3+b_0^2w_1w_2)p^2 \\ 
  +((-3w_1w_2w_3+6b_0)b_1^2-b_0^2w_3-b_0^2w_2-b_0^2w_1)p+b_1^3+b_0^2=0,
\end{multline}
which we write as $ A_3 p^3+A_2 p^2+A_1 p+A_0=0 $, 
where $ A_k$ are functions of $ w_1,w_2,w_3,b_0,b_1$.

Here, we consider the equations
\[
  A_3=A_2=A_1=0.
\]
By computing the Gr\"obner basis of lexicographic order
$ b_1>b_0>w_1>w_2>w_3 $,
we obtain the conditions
\begin{gather*}
   w_3=0,\ w_2=0, \ w_1=0 \\
   \mbox{ or } \quad
    W=(w_2^2-w_1w_2+w_1^2)w_3^2+(-w_1w_2^2-w_1^2w_2)w_3+w_1^2w_2^2=0.  
\end{gather*}
in $ \mathbb{C}[w_1,w_2,w_3] $.
The conditions $ w_k=0 $ ($k=1,2,3$) contradict
the assumption that $R(z)$ has $4$ simple fixed points.
Also we recall that the case that $W=0$ is one excluded 
in the proof of Theorem 1, 
and  actually the condition $W=0$ implies that 
$1/w_1,1/w_2$, and $1/w_3$ form a regular triangle in $\bC$. 
(If $d\geq 4$, we can assume that 
there are neither circles nor lines in $\bC-\{0\}$ 
which contain all non-zero fixed points.)

As before, we consider the conjugate of $R(z)$ by
the translation $L_k(z)= z+ w_k$ for every $k$.
Here we need to consider the case of
$ L_1(z)=z+w_1 $ only, for the other cases are similar. 
Firstly, take the conjugate of $R(z)$ by 
$ L_1(z) $, and secondly
take the conjugate by $ T(z) $, and we see that $R(z)$ changes to
\[
  R^{\#}(z)=\frac{P^{\#}(z)}{Q^{\#}(z)} 
\]
with
\[
  Q^{\#}(z)=b_3^\#z^3 + b_2^\# z^2 + b_1^\# z + b_0^\#,
\]
where 
\[
  \begin{array}{l}
    b^{\#}_3=(w_1^3+(-w_2-w_3)w_1^2+w_3w_2w_1)p^3+(3w_1^2+(-2w_2-2w_3)w_1\\
           \qquad
              +w_3w_2)p^2+(3w_1-w_2-w_3)p+1,\\
    b^{\#}_1=(2w_1b_1+2w_1^2b_2+3w_1^3+(-w_2-w_3)w_1^2+w_3w_2w_1+2b_0)q^2p\\
           \qquad
               +(b_1+2w_1b_2+3w_1^2)q^2, \mbox{ and } \\
    b^{\#}_0=(w_1b_1+w_1^2b_2+w_1^3+b_0)q^3. \\
   \end{array} 
\]
Hence the condition $ b_1^{\#}/b_0^{\#}=-1  $
implies that
\begin{multline*}
   q = \frac{-1}{w_1b_1+w_1^2b_2+w_1^3+b_0}\\
       \times
        \bigl\{\bigl(2w_1b_1+2w_1^2b_2+3w_1^3-(w_2+w_3)w_1^2+w_3w_2w_1
         +2b_0\bigr)p \\
        +b_1+2w_1b_2+3w_1^2\bigr\},
\end{multline*}
and the condition $ b_0^{\#}/b_3^{\#}=1  $
is the equation  
\begin{equation}\label{eq:B}
  B_3 p^3+B_2 p^2+B_1 p+B_0=0
\end{equation}
with
\begin{align*}
  B_3=\,
          & -\biggl\{w_1b_1+w_1^2b_2+2w_1^3
          +(-w_2-w_3)w_1^2+w_3w_2w_1+b_0\biggr\} \\
          & \times
            \biggl\{8w_1^2b_1^2+(16w_1^3b_2+21w_1^4
            +(-5w_2-5w_3)w_1^3+5w_3w_2w_1^2
            +16b_0w_1)b_1\\
          & +8w_1^4b_2^2+(21w_1^5+(-5w_2-5w_3)w_1^4+5w_3w_2w_1^3
            +16b_0w_1^2)b_2\\
          & +14w_1^6+(-7w_2-7w_3)w_1^5+(w_2^2+9w_3w_2+w_3^2)w_1^4\\
          & +(-2w_3w_2^2-2w_3^2w_2+21b_0)w_1^3
            +(w_3^2w_2^2-5b_0w_2-5b_0w_3)w_1^2\\
          & +5b_0w_3w_2w_1+8b_0^2\biggr\},
\end{align*}
\begin{align*}
  B_2=\,
      &(-12w_1^2b_1^3-(48w_1^3b_2+75w_1^4-(14w_2+14w_3)w_1^3
       +13w_3w_2w_1^2+24b_0w_1)b_1^2\\
      & +(-60w_1^4b_2^2+(-186w_1^5 
       +(40w_2+40w_3)w_1^4-38w_3w_2w_1^3-72b_0w_1^2)b_2\\
      & -141w_1^6+(58w_2+58w_3)w_1^5+(-3w_2^2-62w_3w_2-3w_3^2)w_1^4 \\
      & +(6w_3w_2^2+6w_3^2w_2-114b_0)w_1^3
        +(-3w_3^2w_2^2+16b_0w_2+16b_0w_3)w_1^2 \\
      & -14b_0w_3w_2w_1-12b_0^2)b_1-24w_1^5b_2^3
        +(-111w_1^6+(26w_2+26w_3)w_1^5 \\
      & -25w_3w_2w_1^4-48b_0w_1^3)b_2^2+(-168w_1^7+(76w_2+76w_3)w_1^6 \\
      & +(-6w_2^2-86w_3w_2-6w_3^2)w_1^5
        +(12w_3w_2^2+12w_3^2w_2-150b_0)w_1^4 \\
      & +(-6w_3^2w_2^2+28b_0w_2+28b_0w_3)w_1^3
         -26b_0w_3w_2w_1^2-24b_0^2w_1)b_2  \\
      & -84w_1^8+(56w_2+56w_3)w_1^7+(-9w_2^2-73w_3w_2-9w_3^2)w_1^6 \\
      & +(18w_3w_2^2+18w_3^2w_2-114b_0)w_1^5+(-9w_3^2w_2^2 
           +40b_0w_2+40b_0w_3)w_1^4\\
      & -38b_0w_3w_2w_1^3-39b_0^2w_1^2+(2b_0^2w_2+2b_0^2w_3)w_1-b_0^2w_3w_2),
\end{align*}
\begin{align*}
  B_1=\,
      & (-6w_1b_1^3+(-30w_1^2b_2-48w_1^3+(4w_2+4w_3)w_1^2 
        -3w_3w_2w_1-6b_0)b_1^2\\
      & +(-48w_1^3b_2^2+(-150w_1^4+(14w_2+14w_3)w_1^3 
        -12w_3w_2w_1^2-24b_0w_1)b_2 \\
      & -114w_1^5+(20w_2+20w_3)w_1^4-18w_3w_2w_1^3-42b_0w_1^2+(2b_0w_2
              +2b_0w_3)w_1)b_1\\
      & -24w_1^4b_2^3+(-111w_1^5+(13w_2+13w_3)w_1^4-12w_3w_2w_1^3
        -24b_0w_1^2)b_2^2\\
      & +(-168w_1^6+(38w_2+38w_3)w_1^5-36w_3w_2w_1^4-78b_0w_1^3
          +(2b_0w_2+2b_0w_3)w_1^2)b_2\\
      & -84w_1^7+(28w_2+28w_3)w_1^6-27w_3w_2w_1^5
            -60b_0w_1^4+(2b_0w_2+2b_0w_3)w_1^3\\
      & -3b_0^2w_1+b_0^2w_2+b_0^2w_3),  
\end{align*}
and
\begin{align*}
  B_0=\,
      & -b_1^3+(-6w_1b_2-10w_1^2)b_1^2+(-12w_1^2b_2^2-38w_1^3b_2-29w_1^4 
         -2b_0w_1)b_1\\
      & -8w_1^3b_2^3-37w_1^4b_2^2 
         +(-56w_1^5-2b_0w_1^2)b_2-28w_1^6-2b_0w_1^3-b_0^2.
\end{align*}

Now, we consider the equations
\[
    A_3=A_2=A_1=0 \ \mbox{ and } \  B_3=B_2=B_1=0.
\]
By computing the Gr\"obner basis as before,
we obtain the conditions 
\[
   w_3=0,\ w_2=0,\ \mbox{ or } \ w_2-w_3=0,
\]
in $ \mathbb{C}[w_2,w_3] $,
which again gives a contradiction to the assumption.
Therefore, the equation either \eqref{eq:A} or \eqref{eq:B} 
has a solution $ p $.

Thus we have shown the assertion of Theorem 1 for the case that $d=3$.

\

\

\noindent
Masayo Fujimura,
Department of Mathematics, National Defense Academy,
Yokosuka 239-8686, JAPAN (masayo@nda.ac.jp)

\

\noindent
Masahiko Taniguchi,
Department of Mathematics, Nara Women's University,
Nara 630-8506, JAPAN (tanig@cc.nara-wu.ac.jp)

\end{document}